\documentclass{amsart}

\usepackage{amsthm}
\usepackage{amsmath}
\usepackage{amssymb}
\usepackage[all]{xy}
\usepackage{enumerate}

\DeclareMathOperator{\supp}{supp}
\DeclareMathOperator{\dist}{dist}

\usepackage[all]{xy}

\newtheorem{TAS}{Definition}[section]
\newtheorem{TAS hereditary subalgs}{Lemma}[section]
\newtheorem{finding q_0}[TAS]{Lemma}
\newtheorem{move orth projs into a subalg}[TAS]{Lemma}
\newtheorem{K(A_y x Mq) = K(A x Mq)}[TAS]{Lemma}
\newtheorem{proj in A_y x Mq}[TAS]{Lemma}
\newtheorem{sub TAS = TAS}[TAS]{Lemma}
\newtheorem{main theorem}[TAS]{Theorem}
 \newtheorem{classification}{Corollary}[section]

\begin{document}

\markboth{K. R. Strung \& W. Winter}
{Minimal dynamics and $\mathcal{Z}$-stable classification}

\title{MINIMAL DYNAMICS AND $\mathcal{Z}$-STABLE CLASSIFICATION}

\author{KAREN R. STRUNG and WILHELM WINTER}

\address{School of Mathematical Sciences, University of Nottingham\\
University Park, Nottingham NG7 2RD, United Kingdom}

\email{pmxks1@nottingham.ac.uk}

\email{wilhelm.winter@nottingham.ac.uk}
\date{\today}
\thanks{{\it Supported by:} EPSRC First Grant EP/G014019/1}
\maketitle

\begin{abstract}
Let $X$ be an infinite compact metric space, $\alpha : X \to X$ a minimal homeomorphism, $u$ the unitary implementing $\alpha$ in the transformation group $C^{*}$-algebra $C(X) \rtimes_{\alpha} \mathbb{Z}$, and $\mathcal{S}$ a class of separable nuclear $C^*$-algebras that contains all unital hereditary $C^*$-subalgebras of $C^*$-algebras in $\mathcal{S}$. Motivated by the success of tracial approximation by finite dimensional $C^*$-algebras as an abstract characterization of classifiable $C^{*}$-algebras and the idea that classification results for $C^*$-algebras tensored with UHF algebras can be used to derive classification results up to tensoring with the Jiang--Su algebra $\mathcal{Z}$, we prove that $(C(X) \rtimes_{\alpha} \mathbb{Z}) \otimes M_{q^{\infty}}$ is tracially approximately $\mathcal{S}$ if there exists a $y \in X$ such that the $C^*$-subalgebra $(C^*(C(X), uC_0(X \setminus \{y\}))) \otimes M_{q^{\infty}}$ is tracially approximately $\mathcal{S}$. If the class $\mathcal{S}$ consists of finite dimensional \mbox{$C^*$-algebras}, this can be used to deduce classification up to tensoring with $\mathcal{Z}$ for $C^*$-algebras associated to minimal dynamical systems where projections separate tracial states. This is done without making any assumptions on the real rank or stable rank of either $C(X) \rtimes_{\alpha} \mathbb{Z}$ or $C^*(C(X), uC_0(X \setminus \{y\}))$, nor on the dimension of $X$. The result is a key step in the classification of $C^{*}$-algebras associated to uniquely ergodic minimal dynamical systems by their ordered $K$-groups. It also sets the stage to provide further classification results for those \mbox{$C^*$-algebras} of minimal dynamical systems where projections do not necessarily separate traces.
 \end{abstract}

\section{Introduction}
The two subjects of $C^*$-algebras and dynamical systems have long been close allies. On the one hand, dynamical systems provide a rich source of elegant and fundamental examples for $C^*$-algebra theory, see \cite{Cuntz:On}, \cite{EllEva:irrrot} and \cite{Con:Thom}, to name but a few. On the other hand, the techniques of $C^*$-algebras have been used to make progress in distinguishing dynamical systems, most notably in the work of Giordano, Putnam, and Skau on minimal Cantor systems \cite{GioPutSkau:orbit}. The class of $C^*$-algebras associated to actions on infinite compact metric spaces has also played an interesting role in Elliott's classification program for $C^*$-algebras. The classification program was initiated when Elliott showed that the class of approximately finite dimensional (AF) $C^*$-algebras is classified by $K$-theory. Following this success, Elliott conjectured that simple separable nuclear $C^*$-algebras might be classified by a certain $K$-theoretic invariant, now called the Elliott invariant, in the sense that if $A$ and $B$ are two simple separable nuclear $C^*$-algebras with isomorphic invariants, then $A$ and $B$ are $*$-isomorphic as $C^*$-algebras. Moreover, the isomorphism may be chosen in such a way as to induce the isomorphism of Elliott invariants.

Evidence in support of the conjecture was successfully gathered by way of classification results for various classes of $C^*$-algebras, including theorems for those $C^*$-algebras associated to minimal dynamical systems. Many of these results require the presentation of the $C^*$-algebra as a direct limit, showing it is of a form known to be classifiable. For example, the irrational rotation algebras, which are the $C^*$-algebras associated to the dynamical system $(\mathbb{T}, \alpha)$ where $\alpha$ is an irrational rotation of the circle $\mathbb{T}$,  were shown by Elliott and Evans to be A$\mathbb{T}$ algebras with real rank zero~\cite{EllEva:irrrot}, and hence classifiable by their Elliott invariants. Similarly, the crossed products associated to minimal homeomorphisms of the Cantor set were shown, using results by Putnam, to be A$\mathbb{T}$ algebras with real rank zero and thus classifiable \cite{Ell:rrzeroI}.  

In the past decade, efforts have been made to avoid the use of specific direct limit presentations as a means of obtaining classification theorems and more abstract methods have been considered. To this end, Lin introduced the first notion of tracial approximation of $C^*$-algebras with his concept of tracially approximately finite dimensional  (TAF) $C^*$-algebras \cite{Lin:TAF1}. In the case of a simple unital $C^*$-algebra, these may be thought of as being approximated by finite dimensional $C^*$-algebras in trace. Similarly, one may consider $C^*$-algebras that are tracially approximately interval (TAI) algebras. These are $C^*$-algebras that can be approximated by interval algebras, that is, $C^*$-algebras of the form $\bigoplus_{i = 1}^{n} M_{m_i}(C(X_i))$ where $X_i$ is a single point or $X_i = [0,1]$, see \cite{Lin:ttr}. 

Tracial approximation was also considered by Elliott and Niu in \cite{EllNiu:tracial_approx}, where they studied approximation by splitting interval algebras (TASI). In the same paper they consider the concept of tracial approximation in a more general sense, that is, classes of TA$\mathcal{S}$ $C^*$-algebras where $\mathcal{S}$ is an arbitrary class of $C^*$-algebras (see Definition \ref{TAS} below). In particular, they look at which properties of the class $\mathcal{S}$ pass to the class TA$\mathcal{S}$.

The concept of tracial approximation has proven to be extremely useful in providing classification results. One of the most notable results has been by Lin and Phillips in \cite{LinPhi:MinHom}, where classification results for $C^*$-algebras with tracial rank zero were successfully applied to many $C^*$-algebras of minimal dynamical systems.  Let $X$ be an infinite compact metric space, $\alpha : X \to X$ a minimal homeomorphism and $u$ the unitary implementing $\alpha$ in $C(X) \rtimes_{\alpha} \mathbb{Z}$.  Lin and Phillips proved that under certain conditions, it is enough to verify that there exists a point $y \in X$ such that the $C^*$-subalgebra $C^*( C(X), uC_0(X \setminus \{y\}))$ has tracial rank zero. 

In this paper, we generalize the results in \cite{LinPhi:MinHom} by following the strategy for classification up to $\mathcal{Z}$-stability as outlined in \cite{Win:localizingEC}. Here, $\mathcal{Z}$ denotes the Jiang--Su algebra, introduced in \cite{JiaSu:Z}. There are many characterizations of this algebra, both abstract and concrete, which single it out as a universal object playing a role as fundamental as that of the Cuntz algebra $\mathcal{O}_{\infty}$, cf.\ \cite{RorWin:Z-revisited} and \cite{Win:ssa-Z-stable}. $\mathcal{Z}$-stability (i.e., the property of absorbing $\mathcal{Z}$ tensorially) is an important structural property for $C^{*}$-algebras; it has recently been shown to be closely related to other topological and algebraic regularity properties, such as finite topological dimension and strict comparison of positive elements. It is remarkable that all nuclear $C^{*}$-algebras classified so far by their Elliott invariants are in fact $\mathcal{Z}$-stable. 

In \cite{Win:localizingEC}, the second named author derived classification up to $\mathcal{Z}$-stability from classification results up to UHF stability. The latter are usually much easier to establish, since UHF stability ensures a wealth of projections. For any $q \in \mathbb{N} \setminus \{1\}$, let  $M_{q^{\infty}}$ denote the UHF algebra $\bigotimes_{n = 1}^{\infty} M_{q}$. Our main theorem says that if $\mathcal{S}$ is a class of separable unital  $C^*$-algebras such that the property of being a member of $\mathcal{S}$ passes to unital hereditary $C^*$-subalgebras, then $(C(X) \rtimes_{\alpha} \mathbb{Z}) \otimes M_{q^{\infty}}$ is TA$\mathcal{S}$ whenever there exists a $y \in X$ such that the $C^*$-subalgebra $C^*( C(X), uC_0(X \setminus \{y\})) \otimes M_{q^{\infty}}$ is TA$\mathcal{S}$. We tensor with the UHF algebra to alleviate some of the restrictions on $C(X) \rtimes_{\alpha} \mathbb{Z}$ and $C^*( C(X), uC_0(X \setminus \{y\}))$ that are found in \cite{LinPhi:MinHom} and to take advantage of results showing that classification results for $C^*$-algebras tensored with UHF algebras can be used to derive classification results up to tensoring with the Jiang--Su algebra, as shown in \cite{Win:localizingEC}, \cite{Lin:localizingECappendix} and \cite{LinNiu:KKlifting}.  When $\mathcal{S}$ is the set of finite dimensional \mbox{$C^*$-algebras}, our result shows that the $C^*$-algebras associated to minimal dynamical systems of infinite compact metric spaces whose projections separate tracial states are classified by their $K$-theory, up to tensoring with $\mathcal{Z}$.  Toms and the second named author show in \cite{TomsWinter:minhom} (see also \cite{TomsWinter:PNAS}) that in the case where the base space has finite topological dimension the crossed product is $\mathcal{Z}$-stable. Together with the results of the present paper this completes the classification of $C^{*}$-algebras associated to uniquely ergodic minimal finite dimensional dynamical systems by their ordered $K$-groups. We are optimistic that our result will also help provide further classification results for those $C^*$-algebras of minimal dynamical systems of infinite compact metric spaces where projections do not necessarily separate traces.

As an application, we obtain that the $C^{*}$-algebras associated to uniquely ergodic minimal homeomorphisms of odd spheres as considered by Connes in \cite{Con:Thom} are all isomorphic. In the smooth case, this was already shown in \cite{Winter:dr-Z-stable}, using  the inductive limit decomposition of \cite{LinPhi:mindifflimits}. The latter is technically very advanced; our method provides a somewhat easier path since the inductive limit structure of the subalgebras $C^*( C(X), uC_0(X \setminus \{y\}))$ can be established with much less effort.  In the not necessarily uniquely ergodic case, one expects the space of tracial states to be the classifying invariant; building on our present results, this will be pursued in a subsequent article.

The paper is organized as follows. In Section~\ref{Preliminaries} we give the definition for a \mbox{$C^*$-algebra} that is TA$\mathcal{S}$ and outline the strategy for arriving at our main theorem in Section~\ref{Strategy}. In Section~\ref{Main results} we state and prove our main technical results and in Section~\ref{Applications} we apply these to derive classification results. We also discuss some examples and outline the strategy for further results relating to the classification of transformation group $C^{*}$-algebras.

\section{Preliminaries} \label{Preliminaries}

We begin with a definition of what is meant by tracial approximation, cf.\ \cite{Lin:amenableC*-algebras} and \cite{EllNiu:tracial_approx}. 

\begin{TAS}\label{TAS}
Let $\mathcal{S}$ denote a class of separable unital  $C^*$-algebras. Let $A$ be a simple unital $C^*$-algebra. Then $A$ is {\em tracially approximately\/} $\mathcal{S}$ (or $\mathrm{TA}\mathcal{S}$) if the following holds. For every finite subset $\mathcal{F} \subset A$, every $\epsilon > 0$, and every nonzero positive element $c \in A$, there exists a projection $p \in A$ and a unital $C^*$-subalgebra $B \subset pAp$ with $1_{B}=p$ and $B \in \mathcal{S}$ such that:
\begin{enumerate}
\item[\textup{(i)}] $\| pa - ap \| < \epsilon$ for all $a \in \mathcal{F}$,
\item[\textup{(ii)}] $\dist(pap, B) < \epsilon$ for all $a \in \mathcal{F}$,
\item[\textup{(iii)}] $1_{A}-p$ is Murray--von Neumann equivalent to a projection in $\overline{cAc}$.
\end{enumerate}
\end{TAS}

For a compact metric space $X$ and a minimal homeomorphism $\alpha: X \to X$, put $A = C(X) \rtimes_{\alpha} \mathbb{Z}$. We denote $$A_{\{y\}} = C^*( C(X), uC_0(X \setminus \{y\})).$$
$A_{\{y\}}$ is a unital $C^{*}$-subalgebra of $A$, a generalization of those introduced by Putnam in \cite{Putnam:MinHomCantor}. This algebra carries much of the information contained in $A$ while at the same time is significantly more tractable. In particular, by Theorem 4.1(3) of \cite{Phi:CancelSRDirLims}, its $K_{0}$-group is isomorphic to that of $A$, and it can be written as an inductive limit of subhomogeneous algebras in a straightforward manner, see Section 3 of \cite{LinQPhil:KthoeryMinHoms}. There are natural bijections between the set of $\alpha$-invariant probability measures on $X$, the set of tracial states on $A$ and the set of tracial states on $A_{\{y\}}$ (\cite{LinQPhil:KthoeryMinHoms}, Theorem 1.2).

We recall the notion of strict comparison of positive elements for a $C^*$-algebra $A$.  For two positive elements $a, b \in A$, write $a \lesssim b$ if there exists $r_j \in A$ such that $\lim_j r_j b r_j^* = a$. For any $C^*$-algebra $A$, denote by $T(A)$ the tracial state space of $A$. For $\tau \in T(A)$, define
$$ d_{\tau}(a) = \lim_{n \to \infty} \tau(a^{1/n}) \quad (a \in A_+).$$
If $A$ is exact, then the set $\{d_{\tau} \mid \tau \in T(A)\}$ coincides with the set of normalized lower semicontinuous dimension functions of $A$ (\cite{Ror:uhfII} and \cite{Haa:quasi}). If $d_{\tau}(a) < d_{\tau}(b)$ for all $\tau \in T(A)$ implies $a \lesssim b$, then we say that $A$ has \emph{strict comparison (of positive elements)}.  It was recently shown by Toms (\cite{Toms:CompSmoothDyn}, Corollary 5.5) that if $X$ is a compact smooth connected manifold and $\alpha : X \to X$ a diffeomorphism, then $C(X) \rtimes_{\alpha} \mathbb{Z}$ has strict comparison. His result relies on the direct limit structure of $C(X) \rtimes_{\alpha} \mathbb{Z}$ given in \cite{LinPhi:mindifflimits}; such a structure is not known in the general case. However, because we have tensored with a UHF algebra, we are able to use results of R\o rdam to circumvent this issue. For a compact metric space $X$ and a minimal homeomorphism $\alpha: X \to X$, let $A = C(X) \rtimes_{\alpha} \mathbb{Z}$. Then, for any $q \in \mathbb{N} \setminus \{1\}$ and any $y \in X$,  the $C^*$-algebras $A \otimes M_{q^{\infty}}$ and $A_{\{y\}} \otimes M_{q^{\infty}}$ have strict comparison by Theorem 5.2 of \cite{Ror:uhfII}.

\section{Strategy} \label{Strategy}

Our aim is to show that when $A = C(X) \rtimes_{\alpha} \mathbb{Z}$ is the $C^*$-algebra arising from a minimal dynamical system of an infinite compact metric space such that the subalgebra $A_{\{y\}} \otimes M_{q^{\infty}} \subset A \otimes M_{q^{\infty}}$ is TA$\mathcal{S}$, then $A \otimes M_{q^{\infty}}$ is TA$\mathcal{S}$. One can generally only expect such a passage from $A_{\{y\}}$ to $A$ after tensoring with a UHF algebra. We also need to assume the class $\mathcal{S}$ of Definition~\ref{TAS} to be closed with respect to unital hereditary $C^{*}$-subalgebras. 

Our result (and method) is a generalization of the tracial rank zero case (without tensoring with a UHF algebra), which was obtained by Lin and Phillips in \cite{LinPhi:MinHom}. A key lemma for their proof is showing that the finite dimensional $C^*$-subalgebra in the definition of tracial rank zero can be replaced by a sufficiently large simple unital $C^*$-subalgebra of tracial rank zero. Since we are concerned with $A \otimes M_{q^{\infty}}$, we show that for any simple unital $C^*$-algebra $A$, after tensoring with the UHF algebra, $A \otimes M_{q^{\infty}}$ is TA$\mathcal{S}$ when the subalgebra $B \in \mathcal{S}$ in Definition~\ref{TAS} is replaced by a simple unital subalgebra of $A \otimes M_{q^{\infty}}$ that is TA$\mathcal{S}$. With the aim of using this result (Lemma~\ref{sub TAS = TAS} below), the key step is finding a suitable projection $p$ in the $C^*$-subalgebra $A_{\{y\}} \otimes M_{q^{\infty}}$. Under the assumption that $\mathcal{S}$ is closed under taking unital hereditary $C^*$-subalgebras, the unital hereditary $C^*$-subalgebra \mbox{$B := p(A_{\{y\}} \otimes M_{q^{\infty}}) p$} is also TA$\mathcal{S}$; this follows from Lemma 2.3 of \cite{EllNiu:tracial_approx}. The idea then is to find a projection such that $B$ will be nearly invariant under conjugation by $u \otimes 1_{M_{q^{\infty}}}$ and large enough for us to conclude that all of $A \otimes M_{q^{\infty}}$ is TA$\mathcal{S}$. This is achieved in Lemma~\ref{proj in A_y x Mq}. 

We impose no additional requirements on $A_{\{y\}} \otimes M_{q^{\infty}}$ in Lemma~\ref{sub TAS = TAS}. Notice in particular, that by virtue of having the simple stably finite $C(X) \rtimes_{\alpha} \mathbb{Z}$  tensored with a UHF algebra, we do not require that the $C^*$-subalgebra $A_{\{y\}}$ has stable rank one, as is necessary for Lemma 4.2 of \cite{LinPhi:MinHom}; this follows from Corollary 6.6 of \cite{Ror:uhf}. In addition, Lemma 4.2 of \cite{LinPhi:MinHom} requires the assumption that $A_{\{y\}}$ has real rank zero. This is used to pick out an initial small projection using Lemma 4.1 of~\cite{LinPhi:MinHom}. This projection is used to find a sequence of orthogonal projections, where conjugation by $u \otimes 1_{M_{q^{\infty}}}$ acts as a shift, which are then perturbed using Berg's technique \cite{Berg:shiftApprox}. When $\mathcal{S}$ is the set of finite dimensional $C^*$-algebras, then TA$\mathcal{S}$ is simply the class of $C^*$-algebras with tracial rank zero, in which case real rank zero is automatic by Theorem 3.4 of \cite{Lin:TAF1}. However if we are interested in the class of TAI algebras, for example, we can no longer assume this to be the case, as was shown in \cite{Lin:simpleNuclearTR1}.  Thus in an effort to allow the class $\mathcal{S}$ to be as general as possible, we remove the assumption that $A \otimes M_{q^{\infty}}$ has real rank zero.  
Our Lemma~\ref{finding q_0} allows us to find an initial small projection while avoiding the real rank zero restriction.  We then find a projection $p \in A_{\{y\}} \otimes M_{q^{\infty}}$ satisfying the same conditions, (i) -- (iii) of  Lemma \ref{sub TAS = TAS}. Here we do not get an initial sequence of projections all lying in $A_{\{y\}} \otimes M_{q^{\infty}}$.  However the projections lie close enough to this $C^*$-subalgebra that we are able to push them inside. After applying Berg's technique, we once again end up outside $A_{\{y\}} \otimes M_{q^{\infty}}$, but close enough to push the resulting loop inside, eventually ending up with the desired result. 

\section{Main Results} \label{Main results}

\begin{finding q_0} \label{finding q_0}
Let $X$ be an infinite compact metric space and $\alpha : X \to X$ a minimal homeomorphism. Let $y \in X$, and set $A_{\{y\}} = C^*(C(X), uC_0(X \setminus \{y\}))$, where $u$ is the unitary in $C(X) \rtimes_{\alpha} \mathbb{Z}$ implementing $\alpha$. Let $q \in \mathbb{N} \setminus \{1\}$. 

Then for any $\eta > 0$ and any open set $V \subset X$ containing $y$, there exists an open set $W \subset V$ with $y \in W$, functions $g_0 \in C_0(W), g_1 \in C_0(V)$, $0 \leq g_0, g_1 \leq 1$ and a projection $q_0 \in \overline{C_0(V) A_{\{y\}}C_0(V)} \otimes M_{q^{\infty}}$ such that
\[ g_0(y) = 1, \quad g_1 |_W = 1, \quad and \quad \| q_0 (g_1 \otimes 1) - g_1 \otimes 1 \| \leq \eta. \]
\end{finding q_0}

\begin{proof}
We claim that there is a non-zero projection in $\overline{C_0(V) A_{\{y\}} C_0(V)} \otimes M_{q^{\infty}}$.

The set $V$ is non-empty since $y \in V$. Thus $C_0(V)$ is non-zero and hence we can find a non-zero positive contraction in $(C_0(V) A_{\{y\}} C_0(V)) \otimes M_{q^{\infty}}$, call it $e$. Since $A_{\{y\}}$ is simple by Proposition 2.5 of \cite{LinPhi:MinHom}, so is $A_{\{y\}} \otimes M_{q^{\infty}}$. Thus every tracial state $\tau \in T(A_{\{y\}} \otimes M_{q^{\infty}})$ is faithful, and in particular we have $\tau(e) > 0$ for every tracial state $\tau$. Since $A_{\{y\}} \otimes M_{q^{\infty}}$ is unital, $T(A_{\{y\}} \otimes M_{q^{\infty}})$ is compact. Thus $\min_{\tau \in T(A_{\{y\}} \otimes M_{q^{\infty}})} \tau(e) > 0$. Furthemore, $d_{\tau} (e) > \tau(e)$ so the previous observations imply that $\min_{\tau \in T(A_{\{y\}} \otimes M_{q^{\infty}})} d_{\tau}(e) > 0$. 

Since $A_{\{y\}} \otimes M_{q^{\infty}}$ has projections that are arbitrarily small in trace, there is a projection $p \in A_{\{y\}} \otimes M_{q^{\infty}}$ satisfying $$\max_{\tau \in T(A_{\{y\}} \otimes M_{q^{\infty}})} \tau (p) < \min_{\tau \in T(A_{\{y\}} \otimes M_{q^{\infty}})} d_{\tau} (e).$$ 

By the above, for the projection $p$ and any $\tau \in T(A_{\{y\}} \otimes M_{q^{\infty}})$ we have $d_{\tau} (p) = \tau(p) < d_{\tau} (e)$, so by strict comparison there are $x_n \in A_{\{y\}} \otimes M_{q^{\infty}}$ with $x_n e x_n^* \to p$. Let $$a_n = e^{1/2}x_n^* x_n e^{1/2} \in (C_0(V) A_{\{y\}} C_0(V)) \otimes M_{q^{\infty}}.$$ Then $a_n$ is self-adjoint and $$\| a_n - a_n^2\| \to 0.$$
Disregarding any $a_n$ such that $\| a_n - a_n^2\| \geq 1/4$, we obtain a sequence of projections $b_n$ satisfying $\| b_n - a_n\| \leq 2 \|a_n - a_n^2\| \to 0$ (Lemma 2.5.5 of \cite{Lin:amenableC*-algebras}). Thus we obtain, for large enough $n$, a projection $b = b_n$ contained in $\overline{C_0(V) A_{\{y\}} C_0(V)} \otimes M_{q^{\infty}}$, proving the claim. Moreover, $b$ is Murray--von Neumann equivalent to $p$, so $\min_{\tau} \tau(b) = \min_{\tau} \tau(p)$.

Let $W$ be an open set contained in $V$ such that $y \in W$ and small enough so that for every function $f \in C_0(W)$ with $0 \leq f \leq 1$ we have $d_{\tau}(f \otimes 1_{M_{q^{\infty}}}) \leq \frac{1}{2} \min_{\tau} \tau(b)$ for every $\tau \in T(A_{\{y\}} \otimes M_{q^{\infty}})$. Choose $g_0, g_1 \in C_0(W)$ such that $0 \leq g_0, g_1 \leq 1$, $g_0(y) = 1$ and $g_1 g_0 = g_0$. Then $d_{\tau} (g_1 \otimes 1) < d_{\tau}(b)$ for every $\tau \in T(A_{\{y\}} \otimes M_{q^{\infty}})$, and so by the comparison of positive elements we have $(g_1 \otimes 1) \lesssim b$ in $A_{\{y\}} \otimes M_{q^{\infty}}$ and hence also in $\overline{C_0(V) A_{\{y\}} C_0(V)} \otimes M_{q^{\infty}}$.  Since $A_{\{y\}} \otimes M_{q^{\infty}}$ has stable rank one and $\overline{C_0(V) A_{\{y\}} C_0(V)} \otimes M_{q^{\infty}}$ is a full hereditary $C^*$-subalgebra of $A_{\{y\}} \otimes M_{q^{\infty}}$, it also has stable rank one by Theorem 3.6 of \cite{Rfl:sr} with Theorem 2.8 of \cite{Bro:stabher}.  By Proposition 2.4 of \cite{Ror:uhfII}, for $\eta/2 >0$ there is a unitary $v$ in $(\overline{C_0(V) A_{\{y\}} C_0(V)} \otimes M_{q^{\infty}})^+$  (the unitization of $\overline{C_0(V) A_{\{y\}} C_0(V)} \otimes M_{q^{\infty}}$) such that $(g_1 \otimes 1 - {\eta/2})_+ \leq v b v^*$ in  $(\overline{C_0(V) A_{\{y\}} C_0(V)} \otimes M_{q^{\infty}})^{+}$, and hence in $\overline{C_0(V) A_{\{y\}} C_0(V)} \otimes M_{q^{\infty}}$. Put $q_0 = vbv^*$. Then
\begin{eqnarray*}
\| q_0 (g_1 \otimes 1) - (g_1 \otimes 1)\| &< & \| q_0  (g_1 \otimes 1 - \eta/2)_+ - (g_1 \otimes 1) \| + \eta/2\\
 &=& \|(g_1 \otimes 1 - \eta/2)_+ - (g_1 \otimes 1) \| + \eta/2 \\
 &<& \eta. 
 \end{eqnarray*}
\end{proof}

We will use the previous lemma to choose an initial projection in $A_{\{y\}} \otimes M_{q^{\infty}}$. However, since this projection actually only approximates the properties we would like it to have, we require the following easy lemma that pushes orthogonal projections into a $C^*$-subalgebra. The proof is straightforward and hence omitted.

\begin{move orth projs into a subalg} \label{move orth projs into a subalg}
Given $\epsilon > 0$ and a positive integer $n$, there is a $\delta > 0$ with the following property.
Let $A$ be a $C^*$-algebra, $B$ a $C^*$-subalgebra of $A$. Suppose that $p_1, \dots, p_n$ are mutually orthogonal projections in $A$, the first $k$, $0 \leq k \leq n$, of which are contained in B, and that $a_{k+1}, \dots, a_n$ are self adjoint elements of $B$ such that
\[ \| p_i - a_i \| < \min(1/2, \delta), \quad i = k+1, \dots, n. \]
Then there are mutuallly orthogonal projections $q_1, \dots, q_n$ in $B$, where $q_i = p_i$ for $1 \leq i \leq k$, and for $k+1 \leq i \leq n$ we have
\[ \| q_i - p_i \| < \epsilon. \]
Moreover, if $A$ is unital then there are unitaries $u_i \in A$ such that $q_i = u_ip_iu^*_i .$
\end{move orth projs into a subalg}

{\begin{K(A_y x Mq) = K(A x Mq)} \label{K(A_y x Mq) = K(A x Mq)}
Let $X$ be an infinite compact metric space with a minimal homeomorphism $\alpha: X \to X$. Let $A = C^*(X) \rtimes_{\alpha} \mathbb{Z}$ and $A_{\{y\}} = C^*(C(X), uC_0(X \setminus \{y\}))$. Then $$K_0(A_{\{y\}} \otimes M_{q^{\infty}}) \cong K_0(A \otimes M_{q^{\infty}})$$
as ordered groups, with the isomorphism induced by the inclusion $\iota : A_{\{y\}} \to A$.
\end{K(A_y x Mq) = K(A x Mq)}

\begin{proof} Since the $K_1$-group of the UHF algebra $M_{q^{\infty}}$ is $\{ 0 \}$, we have $$K_*(B) \otimes K_*(M_{q^{\infty}}) = (K_0 (B) \otimes K_0(M_{q^{\infty}})) \oplus (K_1(B) \otimes K_0(M_{q^{\infty}}))$$
where $B$ denotes $A$ or  $A_{\{y\}}$.

Let $\iota : A_{\{y\}} \to A$ be the inclusion map. We have the associated inclusion map $\iota \otimes \mathrm{id}_{M_{q^{\infty}}} : A_{\{y\}} \otimes M_{q^{\infty}} \to A \otimes M_{q^{\infty}}$ whence we get the group homomorphisms $K_0(\iota) : K_0(A_{\{y\}}) \to K_0(A)$ and $K_0(\iota \otimes \mathrm{id}_{M_{q^{\infty}}}) : K_0(A_{\{y\}} \otimes M_{q^{\infty}}) \to K_0(A \otimes M_{q^{\infty}})$.  Since  $K_*(M_{q^{\infty}})$ is torsion free, by the K\"unneth Theorem for Tensor Products, we have homomorphisms 
$$\alpha_1 : K_*(A_{\{y\}}) \otimes K_*(M_{q^{\infty}}) \to K_*(A_{\{y\}} \otimes M_{q^{\infty}})$$
and  $$\alpha_2 :   K_*(A) \otimes K_*(M_{q^{\infty}}) \to K_*(A \otimes M_{q^{\infty}}),$$
which are of degree $0$, thus giving maps of the $K_0$ groups, which we will also call $\alpha_1$ and $\alpha_2$. We get the following diagram

\begin{displaymath}
\xymatrix{ 0 \ar[r] & K_0(A_{\{y\}}) \otimes K_0(M_{q^{\infty}}) \ar[d]_{K_0(\iota) \otimes \mathrm{id}_{K_0(M_{q^{\infty}})}} \ar[r]^{\quad \alpha_1} & K_0(A_{\{y\}} \otimes M_{q^{\infty}}) \ar[r] \ar[d]^{K_0(\iota \otimes \mathrm{id}_{M_{q^{\infty}}})}& 0 \\
0 \ar[r] & K_0(A) \otimes K_0(M_{q^{\infty}})  \ar[r]^{\quad \alpha_2} & K_0(A \otimes M_{q^{\infty}}) \ar[r] & 0 }
 \end{displaymath}\\

\noindent
which commutes by naturality. It follows from Theorem 4.1(3)  of  \cite{Phi:CancelSRDirLims} that the map $K_0(\iota) \otimes \mathrm{id}_{K_0(M_{q^{\infty}})}$ is an isomorphism. We conclude that $K_0(\iota \otimes \mathrm{id}_{M_{q^{\infty}}})$ is also an isomorphism. 

Clearly $K_0(\iota \otimes \mathrm{id}_{M_{q^{\infty}}})([1 \otimes 1_{M_{q^{\infty}}}]) = [1 \otimes 1_{M_{q^{\infty}}}]$, and thus preserves order units. It remains to show $K_0(\iota \otimes \mathrm{id}_{M_{q^{\infty}}})$ preserves the order structure and hence is an order isomorphism.

Let $\eta \in K_0(A \otimes M_{q^{\infty}})_+$. Then there are projections $p$ and $q$ in $M_{\infty}(A_{\{y\}} \otimes M_{q^{\infty}})$ such that $K_0(\iota \otimes \mathrm{id}_{M_{q^{\infty}}})([p] - [q]) = \eta$. We show that $[p] - [q] \in K_0(A_{\{y\}} \otimes M_{q^{\infty}})_+$. Let $\tau'$ be the unique tracial state on $M_{q^{\infty}}$. Then, by Theorem 1.2 (4) of \cite{LinQPhil:KthoeryMinHoms}, for $\sigma \in T(A_{\{y\}} \otimes M_{q^{\infty}})$, we have $\sigma = (\tau \circ \iota) \otimes \tau'$ for some $\tau \in T(A)$. By simplicity of $A \otimes M_{q^{\infty}}$, we have $K_0(\tau \otimes \tau')(\eta) > 0$. It follows that $(\tau \otimes \tau')(\iota \otimes \mathrm{id}_{M_{q^{\infty}}})(p - q) > 0$, hence $(\tau \circ \iota) \otimes (\tau' \circ \mathrm{id}_{M_{q^{\infty}}})(p) > (\tau \circ \iota) \otimes (\tau' \circ \mathrm{id}_{M_{q^{\infty}}})(q)$ and finally, $\sigma(p) > \sigma(q)$. Since $A_{\{y\}} \otimes M_{q^{\infty}}$ has strict comparison, we must have $q \lesssim p$, so $[p] - [q] > 0$, as desired.
\end{proof}

The following lemma generalizes Lemma 4.2 of \cite{LinPhi:MinHom}, due to H.~Lin and N.~C.~Phillips.

\begin{proj in A_y x Mq} \label{proj in A_y x Mq} 
Let $X$ be an infinite compact metric space, $\alpha : X \to X$ a minimal homeomorphism, $y \in X$ and $q \in \mathbb{N} \setminus \{1\}$. Let $A = C(X) \rtimes_{\alpha} \mathbb{Z}$ and $A_{\{y\}} = C^*(C(X), uC_0(X \setminus \{y\}))$, where $u$ is the unitary implementing $\alpha$ in $A$. Then, for any finite subset $\mathcal{F} \subset A \otimes M_{q^{\infty}}$ and every $\epsilon > 0$, there is a projection $p$ in $A_{\{y\}} \otimes M_{q^{\infty}}$ such that 
\begin{enumerate}
\item[\textup{(i)}] $\| pa - ap \| < \epsilon$ for all $a \in \mathcal{F}$,
\item[\textup{(ii)}] $\dist (pap, p(A_{\{y\}} \otimes M_{q^{\infty}})p) < \epsilon$ for all $a \in \mathcal{F}$,
\item[\textup{(iii)}] $\tau(1_{A \otimes M_{q^{\infty}}} - p) < \epsilon$ for all $\tau \in T(A \otimes M_{q^{\infty}})$.
\end{enumerate}
\end{proj in A_y x Mq}

\begin{proof} Let $\epsilon> 0$. We first show that there exists a projection satisfying properties (i) -- (iii) of the lemma when $\mathcal{F}$ is assumed to be of the form
\[ \mathcal{F} = ( \mathcal{G} \otimes \{ 1_{M_{q^{\infty}}} \} ) \cup \{ u \otimes 1_{M_q{^{\infty}}} \} \]
where $\mathcal{G}$ is a finite subset of $C(X)$.

Let $N_0 \in \mathbb{N}$ such that $ \pi / (2 N_0) < \epsilon/4$.

Let $\delta_0 > 0$ with $\delta_0 < \epsilon/4$ and sufficiently small so that for all $g \in \mathcal{G}$ we have $\| g(x_1) - g(x_2)\| < \epsilon/ 8$ as long as $d(x_1, x_2) < 4 \delta_0$.

Choose $\delta > 0$ with $\delta < \delta_0$ and such that $d(\alpha^{-n}(x_1), \alpha^{-n}(x_2)) < \delta_0$ whenever $d(x_1, x_2) < \delta$ and $0 \leq n \leq N_0$.

Since $\alpha$ is minimal, there is an $N > N_0 + 1$ such that 
\[ d(\alpha^N(y), y) < \delta.\]
Let $R \in \mathbb{N}$ be sufficiently large so that 
\[R > (N + N_0 + 1 )/ \min(1, \epsilon).\] 

Minimality of $\alpha$ also implies that there is an open neighbourhood $U$ of $y$ such that 
\[ \alpha^{-N_0}(U), \alpha^{-N_0 + 1}(U), \dots, U, \alpha(U), \dots, \alpha^R(U) \]
are all disjoint. Making $U$ smaller if necessary, we may assume that each $\alpha^n(U)$, $-N_0 \leq n \leq R$ has diameter less than $\delta$. To apply Berg's technique, we only need $U_n$ for $-N_0 \leq n \leq N$, however we require $R$ to be larger in order to satisfy property (iii) of the lemma.

Let $\lambda = \max \{ \| g \| \mid g \in \mathcal{G} \}$, and choose
\[ 0< \epsilon_0 < \min (1/2, \epsilon / (2(N +  3 N_0 + 1)), \epsilon/(32 N  (\lambda + \epsilon/4)), \]
\[
0 < \epsilon_{1}<\min (\epsilon_{0}/8, \delta_{\epsilon_{0}, N}/16)
\]
and 
\[ 0 < \eta < \min ( 2\epsilon,  \delta_{\epsilon_1, N_0 + N + 1} )\]
where $\delta_{\epsilon_1, N_0 + N + 1}$ is given by Lemma~\ref{move orth projs into a subalg} with respect to $\epsilon_0$ and  $N_0 + N +1$ in place of $\epsilon$ and $n$, respectively; similarly for $\delta_{\epsilon_{0},N}$. \\

Let $f_0 : X \to [0, 1]$ be continuous with $\supp(f_0) \subset U$, and $f_0 |_V = 1$ for some open set $V \subset U$ containing $y$.\\

By Lemma~\ref{finding q_0}, there is an open set $W \subset V$ containing $y$, functions $g_0 \in C_0(W)$, $g_1 \in C_0(V)$, $0 \leq g_0, g_1 \leq 1$ and a projection $q_0 \in \overline{C_0(V) A_{\{y\}} C_0(V)} \otimes M_{q^{\infty}}$ such that 
\[ g_0(y) = 1, g_1 |_W = 1\text{ and } \| q_0 (g_1 \otimes 1) - g_1 \otimes 1 \| < \eta/2.\]

Consequently, $(f_0 \otimes 1) q_0 = q_0 = q_0 (f_0 \otimes 1)$ and $\| q_0 (g_0 \otimes 1) - g_0 \otimes 1 \| < \eta /2$. \\

For $-N_0 \leq n \leq N$, set
\[ q_n = (u^n \otimes 1) q_0 (u^{-n} \otimes 1), \quad  \quad f_n = u^n f_0 u^{-n} = f_0 \circ \alpha^{-n} \quad \text{ and } \quad U_n= \alpha^n(U). \]
Then $\supp(f_n) \subset U_n$ and
\[ (f_n \otimes 1) q_n = ((u^n f_0 u^{-n}) \otimes 1)(u^n \otimes 1) q_0 (u^{-n} \otimes 1) = (u^n \otimes 1)  (f_0 \otimes 1) q_0 (u^{-n} \otimes 1) = q_n. \]
Similarly, $q_n (f_n \otimes 1) = q_n.$
Since the $f_n$ have disjoint support, it follows that the projections
\[ q_{-N_0}, \dots , q_{-1}, q_0, q_1, \dots, q_N \]
are mutually orthogonal.\\

We claim that $q_{-N_0}, \dots , q_{-1}, q_0 \in A_{\{y\}} \otimes M_{q^{\infty}}$ and that there are self-adjoint $c_1, \dots, c_N \in A_{\{y\}} \otimes M_{q^{\infty}}$ such that $\| q_n - c_n \| < \eta$ for $1 \leq n \leq N$.\\

Let $1 \leq n \leq N_0$ and consider $q_{-n}$. We have $(uf_{-n} \otimes 1) \in A_{\{y\}} \otimes M_{q^{\infty}}$ for $1 \leq n \leq N_0$ since $U_{-n} \cap U_0 = \emptyset$. Let $a_n = f_0^n u^n \otimes 1$.
Then 
\begin{eqnarray*}
a_n = f_0^n u^n \otimes 1 &=& (u u^{-1} f_0 u^2 u^{-2} f_0 u^3 u^{-3} \cdots u^n u^{-n} f_0 u^n) \otimes 1\\
&=& (uf_{-1} \otimes 1) (u f_{-2} \otimes 1) \cdots (uf_{-n} \otimes 1) \in A_{\{y\}} \otimes M_{q^{\infty}}.
\end{eqnarray*}

 From this it follows that 
 \begin{eqnarray*}
 q_{-n} &=& (u^{-n} \otimes 1) q_0 (u^n \otimes 1)\\
 &=& (u^{-n} \otimes 1) (f_0^n \otimes 1) q_0 (f_0^n \otimes 1)(u^n \otimes 1)\\
 &=& a_n^* q_0 a_n \in A_{\{y\}} \otimes M_{q^{\infty}}.
 \end{eqnarray*} 
 
Note that $q_0 ( g_0 \otimes 1) = q_0 (g_1 g_0 \otimes 1)$, since $g_1 |_W = 1$ and $g_0 \in C_0(W)$. Thus
\begin{eqnarray*}
\|  (g_1 g_0 \otimes 1)  - q_0 (g_0 \otimes 1) \| &=& \left\| (  g_1 \otimes 1-  q_0 (g_1 \otimes 1) ) \right (g_0 \otimes 1) \| \\
&<& \eta/2.
\end{eqnarray*}
Also, $g_1 f_0 = g_1$ since $f_0 |_V = 1$ and $g_1 \in C_0(V)$. Hence
\begin{eqnarray*}
\lefteqn{\| (q_0  - g_1 \otimes 1) (f_0 \otimes 1 - g_0 \otimes 1 ) - (q_0 - g_1 \otimes 1) \|} \\
&=& \| q_0 (f_0 \otimes 1) - q_0(g_0 \otimes 1) - (g_1 f_0) \otimes 1 + (g_1 g_0) \otimes 1 - q_0 + g_1 \otimes 1 \| \\
&=& \|  (g_1 g_0 \otimes 1) - q_0( g_0 \otimes 1)\| \\
&<& \eta/2.
\end{eqnarray*}
Since $f_0(y) = 1= g_0(y)$,  we have that $u(f_0 - g_0) \otimes 1 \in A_{\{y\}} \otimes M_{q^{\infty}}$. Set 
\[ c_1 = ( u(f_0 - g_0) \otimes 1) (q_0 - g_1 \otimes 1)(u(f_0 - g_0) \otimes 1)^* + (u g_1 u^* \otimes 1). \]
Then $c_1$ is a self-adjoint element in $A_{\{y\}} \otimes M_{q^{\infty}}$ and
\begin{eqnarray*}
\lefteqn{\| q_1  -  c_1 \|} \\
&=& \| (u \otimes 1) q_0 (u^* \otimes 1) - (u(f_0 - g_0) \otimes 1) (q_0 - g_1 \otimes 1) (u(f_0 - g_0) \otimes 1 )^* \\
&& - (u \otimes 1)(g_1 \otimes 1)(u^* \otimes 1) \| \\ 
&=& \left \| q_0 - ((f_0 - g_0) \otimes 1)(q_0 - g_1 \otimes 1)( (f_0 - g_0) \otimes 1 )  - g_1 \otimes 1  \right \| \\
&\leq& \left \| (q_0 - g_1 \otimes 1) - (q_0 - g_1 \otimes 1)( (f_0 - g_0) \otimes 1 ) \right \| \\ 
&& + \left \| (q_0 - g_1 \otimes 1)( (f_0 - g_0) \otimes 1 )   - ((f_0 - g_0) \otimes 1 ) (q_0 - g_1 \otimes 1) ( (f_0 - g_0 ) \otimes 1) \right \| \\
&\leq& \left \| (q_0 - g_1 \otimes 1) - (q_0 - g_1 \otimes 1) ( (f_0 - g_0) \otimes 1 ) \right \| \\
&&  +  \left \| (q_0 - g_1 \otimes 1) - ( (f_0 - g_0) \otimes 1 ) (q_0 - g_1 \otimes 1)  \right \| \left \|  ( (f_0 - g_0) \otimes 1 ) \right \| \\
& <&  \eta.
\end{eqnarray*}
For $2 \leq n \leq N$, define
\[ c_n  = ((u f_{n-1} \cdots u f_1) \otimes 1 ) c_1((u f_{n-1} \cdots u f_1) \otimes 1 )^* .\]
The $c_n$ are self-adjoint elements in $A_{\{y\}} \otimes M_{q^{\infty}}$ since $f_{n-1}, \dots, f_1$ all vanish at $y$. Furthermore, 
\begin{eqnarray*}
\|q_n - c_n \| &=& \| (u^n \otimes 1 ) q_0 ( u^{-n} \otimes 1) -  c_n \| \\
&=& \left \| ( (u^n f_0^{n-1}) \otimes 1) q_0 ((f_0^{n-1} u^{-n}) \otimes 1) - c_n \right \|  \\
&=& \left \| ( (u^n f_0^{n-1} u^{-1}) \otimes 1) q_1  ((u f_0^{n-1} u^{-n}) \otimes 1 ) -  c_n\right \|  \\
&=& \left\| ((u f_{n-1} \cdots u f_1) \otimes 1 ) q_1((u f_{n-1} \cdots u f_1) \otimes 1 )^* - c_n \right\| \\
&\leq& \| q_1 - c_1 \| \\
&<& \eta.
\end{eqnarray*}
This proves the claim. 

We now apply Lemma~\ref{move orth projs into a subalg} to obtain projections $p_1, \dots, p_N$ in $A_{\{y\}} \otimes M_{q^{\infty}}$ such that
\[ q_{-N_0}, \dots, q_{-1}, q_0, p_1, \dots, p_N \]
are mutually orthogonal, and, for $1 \leq n \leq N$, we have 
\[ \| p_n - q_n \| < \epsilon_1 \]
and unitaries $y_n$ such that 
\[ p_n = y_nq_ny_n^*. \] 

Since $p_n \sim q_n$ and $q_n \sim q_0$, we have $[p_N] = [q_0]$ in $K_0(A \otimes M_{q^{\infty}})$. Since $K_0(\iota \otimes \mathrm{id}_{M_{q^{\infty}}}) : K_0(A_{\{y\}} \otimes M_{q^{\infty}}) \to K_0(A \otimes M_{q^{\infty}})$ is an isomorphism by Lemma~\ref{K(A_y x Mq) = K(A x Mq)}, we also have $[p_N] = [q_0]$ in $K_0(A_{\{y\}} \otimes M_{q^{\infty}})$. Moreover, simplicity of $A_{\{y\}}$ (\cite{LinPhi:MinHom}, Proposition 2.5) implies $A_{\{y\}} \otimes M_{q^{\infty}}$ has stable rank one by Corollary 6.6 of \cite{Ror:uhf}. Thus projections in matrix algebras over $A_{\{y\}} \otimes M_{q^{\infty}}$ satisfy cancellation, and there is a partial isometry $w \in A_{\{y\}} \otimes M_{q^{\infty}}$ such that $w^*w = q_0$ and $w w^* = p_N$.

For $t \in \mathbb{R}$, set
\[ v(t) = \cos(\pi t / 2)(q_0 + p_N) + \sin(\pi t/2)(w  - w^*). \]
Then $v(t)$ is a unitary in the corner $(q_0 + p_N)( A_{\{y\}} \otimes M_{q^{\infty}})(q_0 + p_N)$. The matrix of $v(t)$ with respect to the obvious decomposition is 
\[ \begin{pmatrix}
 \cos(\pi t / 2) & - \sin(\pi t/ 2) \\
  \sin(\pi t/ 2) & \cos(\pi t /2) \\
\end{pmatrix}. \]
For $0 \leq k \leq N_0$, define
\[ w_k = (u^{-k} \otimes 1) v(k/N_0) (u^k \otimes 1). \]
Also, let 
\[ w'_k = (a_k + b_k)^* v(k/N_0)(a_k + b_k) \]
where
\begin{eqnarray*}
a_k &=&  (f_0^k u^k) \otimes 1 = (uf_{-1} \cdots uf_{-k}) \otimes 1 \text{ (as above)}\\
b_k &=& (f_N^k u^k) \otimes 1 = (uf_{N-1} \dots uf_{N-k}) \otimes 1.
\end{eqnarray*}
Both $a_k$ and $b_k$ are in $A_{\{y\}} \otimes M_{q^{\infty}}$, hence \[ w'_k \in A_{\{y\}} \otimes M_{q^{\infty}}.\] We show what $w_k$ is close to $w'_k$.
Define
\[ x_k = (u^{-k} \otimes 1)(q_0 + q_N)v(k/N_0)(q_0 + q_N) (u^k \otimes 1). \]
We have that 
\begin{eqnarray*}
\| w_k - x_k \| &=& \| (u^{-k} \otimes 1)(q_0 + p_N) v(k/N_0) (q_0 + p_N)(u^{k} \otimes 1)\\
&& - (u^{-k} \otimes 1)(q_0 + q_N) v(k/N_0) (q_0 + q_N)(u^{k} \otimes 1)\| \\
&=& \| q_0 v(k/N_0) p_N - q_0 v(k/N_0) q_N + p_N v(k/N_0) q_0 - q_N v(k/N_0) q_0 \\
&& + p_N v(k/N_0) p_N -  q_N v(k/N_0) q_N \| \\
&\leq& 4 \|p_N - q_N \|  \\
&<& 4 \epsilon_1.
\end{eqnarray*}
Also,
\begin{eqnarray*}
\lefteqn{\| w'_k - (a_k  + b_k)^*(q_0 +q_N) v(k/N_0)(q_0 + q_N)(a_k + b_k) \|} \\
 &\leq& \|a_k + b_k\|^2 \|(q_0 + p_N) v(k/N_0)(q_0 + p_N) - (q_0 +q_N) v(k/N_0)(q_0 + q_N) \| \\
& \leq& 4 \| p_N - q_N\| \\
&<&  4 \epsilon_1.
\end{eqnarray*}
But 
\begin{eqnarray*}
\lefteqn{(a_k + b_k)^*(q_0 +q_N) v(k/N_0)(q_0 + q_N)(a_k + b_k)}\\
& =&  (f_0^k u^k + f_N^k u^k)\otimes 1)^*(q_0 + q_N) v(k/N_0)(q_0 + q_N)(f_0^k u^k + f_N^k u^k)\otimes 1) \\
& =& (u^{-k} \otimes 1)(q_0 + q_N) v(k/N_0)(q_0 + q_N)(u^k \otimes 1)\\
& =& x_k.
\end{eqnarray*}
Thus 
\[ \|w_k - w'_k \| < 8\epsilon_1.\]

Comparing $w_k$ to $w_{k+1}$ conjugated by $ u \otimes 1$ we have
$$ \|(u \otimes 1) w_{k+1} (u^{-1} \otimes 1) - w_k \|  =  \|v((k+1)/N_0) -v(k/N_0) \| \leq \pi/(2 N_0) < \epsilon /4$$
for $0 \leq k < N_0$.
Define projections 
\[ e_0 = q_0, e_n = p_n \text{ for } 1 \leq n < N-N_0\]
and 
\[ e_n =  w_{N-n} q_{n-N} w_{N-n}^* \text{ for } N- N_0 \leq n \leq N.\] 

Also define 
\[ d_n = q_n \text{ for } 0 \leq n < N-N_0 \] 
and 
\[ d_n = x_{N-n} q_{n-N} x_{N-n}^* \text{ for } N- N_0 \leq n \leq N. \]

Note that this gives $d_N =x_0 q_0 x_0^* = (q_0 + q_N) v(0) q_0 v(0)^*(q_0 + q_N) = q_0 =  d_0$ and and $e_N = v(0) q_0 v(0)^* = q_0 = e_0$. We also have that the $x_k^* x_l = 0$ when $k \neq l$. This follows from the fact that, if $k \neq l$, then $q_{-k}, q_{N-k}, q_{-l}$ and $q_{N-l}$, $0 \leq k \neq l \leq N_0$, are mutually orthogonal and 
\begin{eqnarray*}
\lefteqn{(q_0 + q_N)(u^k \otimes 1)(u^{-l} \otimes 1)(q_0 + q_N)}\\
&=& (u^k \otimes 1)(q_{-k} + q_{N-k})(q_{-l} + q_{N-l})(u^{-l} \otimes 1)\\
&=& 0.
\end{eqnarray*}
Also, if $0 < m < N-N_0$ and $N- N_0 \leq n \leq N$ then 
\[ q_m (u^{-(N-n)} \otimes 1) (q_0 + q_N) = q_m (q_{-(N-n)} + q_n) (u^{-(N-n)} \otimes 1) = 0,\] and similarly 
\[  (q_0 + q_N) (u^{N-n} \otimes 1) q_m = 0. \]
From this it follows that $d_m d_n = 0$ for $0 \leq m \neq n \leq N$. 

For $1 \leq n \leq N - N_0 - 1$ we have 
\[ \| e_n - d_n \| = \| p_n - q_n \| < \epsilon_1 \]
and
\[ \| e_0 - d_0 \| = 0. \]

If $N - N_0 \leq n \leq N$, then 
\begin{eqnarray*}
\| e_n - d_n \|&=& \| w_{N-n} q_{-(N-n)} w_{N-n}^* - x_{N-n} q_{-(N-n)} x_{N-n}^* \| \\
&=&  \| w_{N-n} q_{-(N-n)} w_{N-n}^* - w_{N-n} q_{-(N-n)} x_{N-n}^* + w_{N-n} q_{-(N-n)} x_{N-n}^* \\
&& - x_{N-n} q_{-(N-n)} x_{N-n}^*\| \\
& \leq& \| w_{N-n}^* - x_{N-n}^* \| + \|w_{N-n} -  x_{N-n}\| \\
& <& 4 \epsilon_1  + 4 \epsilon_1 \\
&=& 8 \epsilon_1 .
\end{eqnarray*}

We now show that conjugating the $d_n$ by $u \otimes 1$ acts approximately  as a cyclic shift.  For $1 \leq n \leq N-N_0 - 1$ we have $(u \otimes 1) d_{n-1} (u \otimes 1)^* = d_n$ since $d_n = q_n$.\\

If $n = N- N_0$, then 
\begin{eqnarray*}
d_{N-N_0} &=& x_{N_0} q_{-N_0} x^*_{N_0} \\
&=& (u^{-N_0} \otimes 1)(q_0 + q_N) v(1) (q_0 + q_N)q_0(q_0 + q_N)v(-1)(q_0 + q_N) (u^{N_0} \otimes 1) \\
&=& (u^{-N_0} \otimes 1)(q_0 + q_N) p_N (q_0 + q_N) (u^{N_0} \otimes 1) \\
&=& (u^{-N_0} \otimes 1)q_N p_N q_N(u^{N_0} \otimes 1).
\end{eqnarray*}
Thus
\begin{eqnarray*}
\lefteqn{\| (u \otimes 1)d_{N- N_0 - 1} ( u^* \otimes 1)  - d_{N-N_0} \|} \\
&=& \|q_{N- N_0} -  (u^{-N_0} \otimes 1)q_N p_N q_N(u^{N_0} \otimes 1) \| \\
&=&  \|  (u^{-N_0} \otimes 1)q_N (u^{N_0} \otimes 1) -  (u^{-N_0} \otimes 1)q_N p_N q_N(u^{N_0} \otimes 1) \| \\
&\leq& \|q_N - p_N \| \\
&<& \epsilon_1  .
\end{eqnarray*}

When $N- N_0 < n \leq N$, first consider what happens to the $e_n$ using the estimation made on the $w_k$ above. We have
\begin{eqnarray*}
\lefteqn{\| (u \otimes 1)  e_{n-1} (u^* \otimes 1) - e_n \|} \\ 
&=& \| (u \otimes 1) w_{N- (n-1)} (u^* \otimes 1) q_{n-N} (u \otimes 1) w_{N- (n-1)}^* (u^* \otimes 1) \\
&&  - w_{N-n} q_{n-N} w_{N-n}^* \|  \\
&\leq& \| (u \otimes 1) w_{N- (n-1)} (u^* \otimes 1) q_{n-N} (u \otimes 1) w_{N- (n-1)}^* (u^* \otimes 1)  - \\
&& (u \otimes 1) w_{N- (n-1)} (u^* \otimes 1) q_{n-N}w_{N-n}^* \| \\
&& + \| (u \otimes 1) w_{N- (n-1)} (u^* \otimes 1) q_{n-N}w_{N-n}^* - w_{N-n} q_{n-N} w_{N-n}^* \| \\
&\leq& \| (u \otimes 1) w_{N- (n-1)}^* (u^* \otimes 1) - w_{N-n}^* \| \\
&& + \| (u \otimes 1) w_{N- (n-1)} (u^* \otimes 1) - w_{N-n} \| \\
&<& \epsilon/ 2.
\end{eqnarray*}
From this we have
\begin{eqnarray*}
\| (u \otimes 1) d_{n-1} (u^* \otimes 1) - d_n \| &<& \| e_{n-1} - d_{n-1} \| + \| e_n - d_n \| + \epsilon /2 \\
&<& 16 \epsilon_1 + \epsilon/2.
\end{eqnarray*}

Now we use the fact that the $w_k$ are almost in $A_{\{y\}} \otimes M_{q^{\infty}}$ to find  projections in $A_{\{y\}} \otimes M_{q^{\infty}}$ that lie close to the $e_n$ and hence also close to the $d_n$. When $0 \leq n \leq N- N_0 - 1$, we have $e_n = p_n \in A_{\{y\}} \otimes M_{q^{\infty}}$. Also, since $e_N = q_0$,  we only need to find projections when $N - N_0 \leq n \leq N- 1$. In this case, we have
\begin{eqnarray*}
 \lefteqn{\| e_n - w'_{N-n} q_{-(N-n)} (w'_{N-n})^{*} \|} \\
 &=& \| w_{N-n} q_{-(N-n)} w_{N-n}^* - w'_{N-n} q_{-(N-n)} (w'_{N-n})^* \| \\
 &<& 16 \epsilon_1. 
 \end{eqnarray*}
Since $w'_{N-n} q_{-(N-n)} (w'_{N-n})^* \in A_{\{y\}} \otimes M_{q^{\infty}}$, by Lemma~\ref{move orth projs into a subalg} we find orthogonal projections $r_n \in A_{\{y\}} \otimes M_{q^{\infty}}$ with 
\[ \| r_n - e_n \| <  \epsilon_0 \text{ and } r_n = z_n e_n z_n^*\]
for unitaries $z_n \in A \otimes M_{q^{\infty}}$. This also implies that 
\[ \| r_n - d_n \| <  \epsilon_0 + 8 \epsilon_1 < 2 \epsilon_0. \] 

For $1 \leq n \leq N - N_0 - 1$ put $r_n = e_n = p_n$ and put $r_N = e_N = q_0$. Then set 
$$r = \sum_{n =1}^N r_n \qquad \text{ and } \qquad p = 1 - r.$$
We verify that the projection $p \in A_{\{y\}} \otimes M_{q^{\infty}}$ satisfies properties (i) -- (iii) of the lemma.\\

Let $d = \sum_{n = 1}^N d_n$. Note that
\[ d - (u \otimes 1) d (u \otimes 1)^* =  \sum_{n = N - N_0}^N ( (u \otimes 1) d_{n - 1} (u \otimes 1)^*  - d_n ).\] 
For $N-N_0 \leq m \neq n \leq N$, we have
\begin{eqnarray*}
\lefteqn{( (u \otimes 1) d_{n-1} (u \otimes 1)^* - d_n)((u \otimes 1) d_{m-1} (u \otimes 1)^* - d_m)} \\
&=&  (u \otimes 1) d_{n-1}d_{m-1} (u \otimes 1)^* - (u \otimes 1) d_{n-1} (u \otimes 1)^* d_m \\
&& - d_n(u \otimes 1) d_{m-1} (u \otimes 1)^* + d_n d_m \\
&=& -(u\otimes 1) x_{N-(n-1)}q_{n-1-N}x_{N-(n-1)}^* (u \otimes 1)^* x_{N-m}q_{m-N}x_{N-m}^* \\
&& - x_{N-n}q_{n-N}x_{N-n}^*(u \otimes 1) x_{N-(m-1)}q_{m-1-N} x_{N-(m-1)}^*(u \otimes 1)^*\\
&=& -(u \otimes 1) x_{N-(n-1)}q_{n-1-N}(u \otimes 1)^* x_{N-n}^* x_{N-m} q_{m-N} x_{N-m}^*\\
&& - x_{N-n}q_{n-N}x_{N-n}^*x_{N-m}(u \otimes 1)q_{m-1-N}x_{N-(m-1)}^*(u \otimes 1)^* \\
&=& 0.
\end{eqnarray*}
Thus the terms in the sum are mutually orthogonal with norm at most $16 \epsilon_1 + \epsilon/2$, hence 
\[ \| d - (u \otimes 1) d (u \otimes 1)^* \| < 16 \epsilon_1 + \epsilon/2 .\]
Now
\begin{eqnarray*}
\lefteqn{\| p - (u \otimes 1) p (u \otimes 1)^* \|} \\
&=& \| ((u \otimes 1) r (u^* \otimes 1) - r) - ((u \otimes 1) d (u^* \otimes 1) - d ) + ((u \otimes 1) d (u^* \otimes 1) - d ) \|\\
& \leq& 2 \|r -d \| +  16 \epsilon_1 + \epsilon/2 \\
&< & \sum_{n=1}^{N-N_0-1} 2  \|p_n - q_n\| + \sum_{m=N-N_0}^{N-1} 2 \|r_m - d_m\| + 16 \epsilon_1+  \epsilon/2 \\
&<& 2(N - N_0 -1) \epsilon_1 + 4 N_0 \epsilon_0 + 16 \epsilon_1 + \epsilon/2 \\
&<& (N-N_0-1) \epsilon_0 + 4 N_0 \epsilon_0 + 2 \epsilon_0 + \epsilon/2\\
&<& \epsilon.
\end{eqnarray*}

Since $g_1(y) = 1$ it follows that $ u(1 - g_1) \otimes 1 \in A_{\{y\}} \otimes M_{q^{\infty}}$. Thus we also have that $p (u\otimes 1)((1 - g_1)\otimes 1 )(1- q_0) p \in A_{\{y\}} \otimes M_{q^{\infty}}$. Note that $p \leq 1- q_0$. Using this and the fact that $\| g_1 \otimes 1 - (g_1 \otimes 1) q_0 \| < \eta/2 < \epsilon$, it follows that
\begin{eqnarray*}
\lefteqn{\| p(u \otimes 1)p - p (u\otimes 1)((1 - g_1)\otimes 1 )(1- q_0) p \|} \\
&=& \| p(u \otimes 1)p - p(u \otimes 1)p + p(u \otimes 1)(g_1 \otimes 1)(1 - q_0) p \|\\
&\leq& \| p(u \otimes 1)(g_1 \otimes 1) p - p(u \otimes 1)(g_1 \otimes 1) q_0 p \| \\
&<& \epsilon.
\end{eqnarray*}
This proves (i) and (ii) for the element $u \otimes 1 \in \mathcal{F}$.\\

Now consider $g \otimes 1 \in \mathcal{F}$, where $g \in C(X)$. \\

Since $d(\alpha^N(y), y) < \delta$, we have $d(\alpha^n(y), \alpha^{n-N}(y)) < \delta_0$ for $N - N_0 \leq n \leq N$. It follows that $U_{n-N} \cup U_n$ has diameter less than $2 \delta + \delta_0 \leq 3 \delta_0$. The function $g \in \mathcal{G}$ varies by at most $\epsilon/8$ on sets of diameter less than $4\delta_0$, and since the sets 
\[ U_1, U_2, \dots, U_{N- N_0 - 1}, U_{N-N_0} \cup U_{-N_0}, U_{N-N_0+1} \cup U_{-N_0 + 1}, \dots, U_N \cup U_0 \]
are open and pairwise disjoint, there is $\tilde{g} \in C(X)$ which is constant on each of these sets and satisfies $\| g - \tilde{g}\| < \epsilon/4$. Let the values of $\tilde{g}$ on these sets be $\lambda_{1}$ on $U_1$ through to $\lambda_{N}$ on $U_N \cup U_0$. \\

For $0 \leq n \leq N- N_0 -1$ we have
\begin{eqnarray*}
\| (f_n \otimes 1) r_n  - r_n \| &=& \| (f_n \otimes 1) r_n  - (f_n \otimes 1) q_n + q_n - r_n \| \\
&\leq& 2 \| q_n - p_n \| \\
&<& 2 \epsilon_1.
\end{eqnarray*}
Thus
\begin{eqnarray*}
\| (\tilde{g} \otimes 1) r_n - \lambda_n \cdot r_n \| & \leq& \| (\tilde{g} \otimes 1) r_n - (\tilde{g} \otimes 1)(f_n \otimes 1) r_n  \|  \\
&& +  \|(\tilde{g} \otimes 1)(f_n \otimes 1) r_n - \lambda_n \cdot r_n  \| \\
&<& 4 \|\tilde{g} \| \epsilon_1.
\end{eqnarray*}

For $N - N_0 \leq n \leq N$, we have that $(f_{n-N} + f_n)x_{N-n} =  x_{N-n}$, since we may write $x_{N-n} = (q_{n-N} + q_n) (u^{n - N} \otimes 1) v((N-n)/N_0 ) (u^{N-n} \otimes 1)(q_{n-N} + q_n)$. Similarly, $x_{N-n}^* (f_{n-N} + f_n) = x_{N-n}^*$. Thus $(f_{n-N} + f_n)d_n = d_n = d_n (f_{n-N} + f_n)$. It follows that
\begin{eqnarray*}
\| (f_{n-N} + f_n) r_n - r_n \| &=&  \|(f_{n-N} + f_n) r_n - (f_{n-N} + f_n) d_n + d_n - r_n \| \\
&<& 4 \epsilon_0.
\end{eqnarray*}
Thus, similar to the above, $\| (\tilde{g} \otimes 1) r_n - \lambda_n \cdot r_n \| < 8 \|\tilde{g} \| \epsilon_0$.\\

Hence
\begin{eqnarray*}
\| (g \otimes 1)p - p(g \otimes 1) \| &<&  \| (\tilde{g} \otimes 1)p- p(\tilde{g }\otimes 1) \| + \epsilon/2\\
&\leq&  \sum_{n=1}^N \|  (\tilde{g} \otimes 1)r_n - \lambda_n \cdot r_n + \lambda_n \cdot r_n - r_n (\tilde{g} \otimes 1) \| +  \epsilon/2 \\
&\leq& 2N (8 \|\tilde{g} \| \epsilon_0) +  \epsilon/2 \\
&<& \epsilon.
\end{eqnarray*}

This shows property (i) of the lemma for $g \otimes 1$, $g \in \mathcal{G}$. The second condition is immediate since $g \otimes 1$ is an element of $A_{\{y\}} \otimes M_{q^{\infty}}$.\\

It remains to verify the third condition.\\

Since the sets $\alpha^{-N_0}(U), \alpha^{-N_0 + 1}(U), \dots, U, \alpha(U), \dots, \alpha^R(U)$ are all disjoint and $R > (N + N_0 + 1) / \min(1, \epsilon)$, it follows that 
$$\sum_{n = -N_0}^R u^{n} f_0 u^{-n} = f_0^{-N_0} + \dots + f_0 + \dots + f_0 ^R \leq 1$$
and hence $\tau_1(f_0) \leq \tau_1(1)/(R+N_0 + 1) < \epsilon/(N + N_0 +1)$ for every $\tau_1 \in T(A)$. Since any $\tau \in T(A \otimes M_{q^{\infty}})$ is of the form $\tau = \tau_1 \otimes \tau_2$ for $\tau_1 \in T(A)$ and $\tau_2$ the unique tracial state on $M_{q^{\infty}}$, we have $\tau(q_0) \leq \tau (f_0 \otimes 1) = \tau_1(f_0) < \epsilon/(N + N_0 +1)$. For $1 \leq n \leq N-N_0 -1$ each $r_n$ is just $q_0$ conjugated by a unitary so $\tau(r_n) = \tau(q_0)$. For $N-N_0 \leq n \leq N$, $r_n = z_n e_n z_n^* = z_n w_{N-n}q_{-(N-n)}w^*_{N-n}z_n^*$. Thus
\begin{eqnarray*}
\tau(r_n) &=& \tau(z_n w_{N-n}q_{-(N-n)}w^*_{N-n}z_n^*) \\
&=& \tau(w_{N-n}q_{-(N-n)}w^*_{N-n})\\
&=& \tau (v ( (N-n)/{N_0})q_0v ((N-n)/N_0)^*) \\
&=& \tau((q_0 +  p_N)q_0 ) \\
&=& \tau(q_0).
\end{eqnarray*}
Thus
\[ \tau(1 - p ) = \sum_{n=1}^N \tau(r_n) = \sum_{n=1}^{N} \tau(q_0)  < N \epsilon/(N + N_0 +1) < \epsilon. \]

This proves the case where $\mathcal{F}$ is of the form $( \mathcal{G} \otimes \{ 1_{M_{q^{\infty}}} \} ) \cup \{ u \otimes 1_{M_{q^{\infty}}} \}$.\\

For the general case, let $\tilde{\mathcal{F}} \subset A \otimes M_{q^{\infty}}$ be a finite subset. Using the identification 
\[ A \otimes M_{q^{\infty}} \cong A \otimes M_{q^r} \otimes M_{q^{\infty}} \cong A \otimes M_{q^{\infty}} \otimes M_{q^r}, \]
for $r \in \mathbb{N}$, we may assume that the finite set is of the form
\[ (\{1_A \}\otimes \{1_{M_{q^{\infty}}} \}  \otimes \mathcal{B} )\cup ( \mathcal{\tilde{G}} \otimes \{1_{M_{q^{\infty}}} \} \otimes  \{1_{M_{q^r}} \} ) \cup (\{u\} \otimes \{1_{M_{q^{\infty}}} \}\otimes \{1_{M_{q^r}} \}) \]
where $r \in \mathbb{N}$,  $\mathcal{B}$ is a finite subset of $M_{q^r}$ and $\mathcal{G}$ is a finite subset of $C(X)$.\\

We may further assume that $1_X = 1_A \in \mathcal{G}$ and also that $1_{M_{q^r}} \in \mathcal{B}$. Then $\mathcal{F} = (\mathcal{G}  \otimes \{ 1_{M_{q^{\infty}}}\}) \cup \{u \otimes 1_{M_{q^{\infty}}} \}$ and $\mathcal{\tilde{F}} = \mathcal{F} \otimes \mathcal{B}$. \\

Let $\epsilon > 0$. By the above, there exists a projection $p \in A_{\{y\}} \otimes M_{q^{\infty}}$ satisfying properties (i) -- (iii) of the lemma for the finite set $\mathcal{F} = \mathcal{G} \otimes \{ 1_{M_{q^{\infty}}} \} \cup \{u \otimes 1_{M_{q^{\infty}}} \}$, with $\epsilon / \max (\{ \|b\| \mid b \in \mathcal{B} \}, 1)$ in place of $\epsilon$. \\
 
 Define $\tilde{p} := p \otimes 1_{M_{q^r}} \in A_{\{y \}} \otimes M_{q^{\infty}} \otimes M_{q^r}$. We now show that $\tilde{p}$ satisfies properties (i) -- (iii) of the lemma for $\mathcal{\tilde{F}}$ and $\epsilon$.
 
 Let $\tilde{a} \in \mathcal{\tilde{F}}$. Then $\tilde{a} = a \otimes b$ for some $a \in \mathcal{F}$ and some $b \in \mathcal{B}$. We have
 \begin{eqnarray*}
 \| \tilde{p} \tilde{a} - \tilde{a} \tilde{p} \| &=& \| (p \otimes 1)( a \otimes b) -  (a \otimes b)(p \otimes 1)\| \\
 &=& \|(pa) \otimes b - (ap) \otimes b \| \\
 &=& \| (pa - ap) \otimes b \| \\
 &=&  \|pa -ap \| \|b \| \\
 &<& \epsilon.
 \end{eqnarray*}
 
By the special case above, for every $a \in \mathcal{F}$, there is some $x \in p(A_{\{y\}} \otimes M_{q^{\infty}})p$ such that $\| pap - x\| < \epsilon / (\max_{b \in \mathcal{B}} \|b \|)$. Thus $x \otimes 1 \in \tilde{p} (A_{\{y\}} \otimes M_{q^{\infty}} \otimes M_{q^r}) \tilde{p}$. It is clear that $\tilde{p}(1 \otimes b) \tilde{p} \in \tilde{p} (A_{\{y\}} \otimes M_{q^{\infty}} \otimes M_{q^r}) \tilde{p}$ for any $b \in \mathcal{B}$, and so $x \otimes b \in  \tilde{p} (A_{\{y\}} \otimes M_{q^{\infty}}\otimes M_{q^r}) \tilde{p}$. It follows that
\begin{eqnarray*}
\| \tilde{p}(a \otimes b) \tilde{p} -  \tilde{p}(x \otimes b) \tilde{p} \| &=& \|\tilde{p}(a \otimes 1)(1 \otimes b) \tilde{p} - \tilde{p}(x \otimes 1)(1 \otimes b) \tilde{p} \| \\
&=& \|\tilde{p}(a \otimes 1) \tilde{p} (1 \otimes b) - \tilde{p}(x \otimes 1) \tilde{p}(1 \otimes b) \| \\
&=& \| (pap - pxp) \otimes 1 \| \|b\| \\
&<& \epsilon.
\end{eqnarray*}

This shows that (i) and (ii) hold. 

To prove (iii), simply observe that $\tau \in  T(A \otimes M_{q^{\infty}}\otimes M_{q^r})$ is of the form $\tau_1 \otimes \tau_2$ where $\tau_1 \in T(A \otimes M_{q^{\infty}})$ and $\tau_2 \in T(M_{q^r})$. Then 
\[ \tau(1 - \tilde{p}) = \tau(1 \otimes 1 - p \otimes 1) = \tau \left( (1-p) \otimes 1 \right) = \tau_1(1-p)\tau_2(1) < \epsilon.  \] 
\end{proof}

For the following lemma, we do not need any assumptions on the properties of the class $\mathcal{S}$ of separable unital $C^*$-algebras. However, for Theorem~\ref{main theorem}, we cut down a TA$\mathcal{S}$ $C^*$-subalgebra $B \subset A \otimes M_{q^{\infty}}$ by a projection; there we must make the additional requirement that the property of being a member of $\mathcal{S}$ passes to unital hereditary subalgebras. 

\begin{sub TAS = TAS} \label{sub TAS = TAS} Let $\mathcal{S}$ be a class of separable unital $C^*$-algebras.  Let $A$ be a simple unital $C^*$-algebra and $q \in \mathbb{N} \setminus \{1\}$. Suppose that for every finite subset $\mathcal{F} \subset A \otimes M_{q^{\infty}}$, every $\epsilon > 0$, and every nonzero positive $c \in A \otimes M_{q^{\infty}}$, there exists a projection $p \in A \otimes M_{q^{\infty}}$ and a simple unital $C^{*}$-subalgebra $B \subset p(A \otimes M_{q^{\infty}})p$ which  is TA$\mathcal{S}$, satisfies $1_{B}=p$ and
\begin{enumerate}
\item[\textup{(i)}] $\| pa - ap \| <  \epsilon$ for all $a \in \mathcal{F}$, 
\item[\textup{(ii)}] $\dist(pap, B) < \epsilon$ for all $a \in \mathcal{F}$,
\item[\textup{(iii)}] $1_{A} - p$ is Murray--von Neumann equivalent to a projection in $\overline{c (A \otimes M_{q^{\infty}}) c}$.
\end{enumerate}
Then $A \otimes M_{q^{\infty}}$ is TA$\mathcal{S}$.
\end{sub TAS = TAS}

\begin{proof} Although a TA$\mathcal{S}$ $C^*$-algebra may not have property (SP), the $C^*$-algebra $A \otimes M_{q^{\infty}}$ always will, since $A \otimes M_{q^{\infty}}$ has strict comparison (cf.\ \cite{Ror:uhfII}) and contains nonzero projections which are arbitrarily small in trace. After noting this, the proof is essentially the same as that of Lemma 4.4 of \cite{LinPhi:MinHom}, replacing the $C^*$-subalgebra of tracial rank zero with the TA$\mathcal{S}$ $C^*$-subalgebra $B$, and replacing the finite dimensional $C^*$-subalgebra with a $C^*$-subalgebra from the class $\mathcal{S}$.
\end{proof}

\begin{main theorem} \label{main theorem} Let $\mathcal{S}$ be a class of separable unital $C^*$-algebras such that the property of being a member of $\mathcal{S}$ passes to unital hereditary $C^*$-subalgebras. Let $X$ be an infinite compact metric space, $\alpha : X \to X$ a minimal homeomorphism, let $u$ be the unitary implementing $\alpha$ in $A := C(X) \rtimes_{\alpha} \mathbb{Z}$ and $q \in \mathbb{N} \setminus \{1\}$. Suppose there is a $y \in X$ such that $A_{\{y\}} \otimes M_{q^{\infty}}$ is TA$\mathcal{S}$. Then $A \otimes M_{q^{\infty}}$ is TA$\mathcal{S}$.
\end{main theorem}

\begin{proof} We show that $A \otimes M_{q^{\infty}}$ satisfies the conditions of Lemma~\ref{sub TAS = TAS}. 

Let $\epsilon > 0$, $\mathcal{F}$ a finite subset of $A \otimes M_{q^{\infty}}$ and a positive nonzero element $c$ in $A \otimes M_{q^{\infty}}$ be given. Use Lemma \ref{proj in A_y x Mq} to find a projection $p \in A_{\{y\}} \otimes M_{q^{\infty}}$ with respect to $\mathcal{F}$, $c$, and $\epsilon_0 = \min( \epsilon, \min_{\tau \in T(A \otimes M_{q^{\infty}})} \tau(c))$. Put $B = p(A_{\{y\}} \otimes M_{q^{\infty}})p$. It is a unital simple $C^*$-subalgebra of $p(A \otimes M_{q^{\infty}})p$ and is TA$\mathcal{S}$ by the assumptions made on $\mathcal{S}$ and Lemma 2.3 of \cite{EllNiu:tracial_approx}.  Conditions (i) and (ii) of Lemma~\ref{sub TAS = TAS} are satisfied by the choice of $p$. Since $\tau(1_{A} - p) < \min_{\sigma \in T(A \otimes M_{q^{\infty}})} \sigma(c) < \tau(c)$ for every tracial state $\tau \in T(A \otimes M_{q^{\infty}})$, it follows from Theorem 5.2(a) of \cite{Ror:uhfII} that $1_{A}-p$ is Murray--von Neumann equivalent to a projection in $\overline{c(A \otimes M_{q^{\infty}})c}$. Thus $A \otimes M_{q^{\infty}}$ is TA$\mathcal{S}$ by Lemma~\ref{sub TAS = TAS}. 
\end{proof}

\section{Classification. Outlook.} \label{Applications}

Recall that for a separable simple unital stably finite nuclear $C^*$-algebra $A$, the Elliott invariant of $A$ is given by 
$$((K_0(A), K_0(A)_+, [1_A]), K_1(A), T(A), r_A : T(A) \to S(K_0(A)))$$
consisting of the ordered $K$-groups, the Choquet simplex of tracial states $T(A)$ and \[r_A : T(A) \to S(K_0(A)),\] the canonical affine map to the state space of $(K_0(A), K_0(A)_+, [1_A])$ given by $r_A(\tau)([p]) = \tau(p)$ \cite{Ror:encyc}. Since we are interested in applying our results to Elliott's classification program, the most immediate application for Theorem~\ref{main theorem} is when $\mathcal{S}$  is the set of finite dimensional $C^*$-algebras, where we are able to apply Lin's classification theorem for $C^*$-algebras of tracial rank zero. A simple unital $C^*$-algebra with tracial rank zero always has real rank zero (\cite{Lin:TAF1}, Theorem 3.4). When $A$ has real rank zero, the map $r_A$ is bijective, and as such the invariant becomes the ordered $K$-theory \cite{Ror:encyc}.  Applying Theorem \ref{main theorem} to this special case, we have the following classification result up to tensoring with the Jiang--Su algebra $\mathcal{Z}$.

\begin{classification} \label{classification}Let $\mathcal{A}$ denote the class of $C^*$-algebras with the following properties.
\begin{enumerate}
\item[\textup{(i)}] $A \in \mathcal{A}$ is of the form $C(X) \rtimes_{\alpha} \mathbb{Z}$ for some infinite compact metric space $X$ and minimal homeomorphism $\alpha$. 
\item[\textup{(ii)}] The projections in $A$ separate $T(A)$.
\end{enumerate}

Let $A, B \in \mathcal{A}$ and suppose there is a graded order isomorphism $ \phi : K_*(A \otimes \mathcal{Z}) \to K_*(B \otimes \mathcal{Z})$.  Then there is a $*$-isomorphism $\Phi: A \otimes \mathcal{Z} \to B \otimes \mathcal{Z}$ inducing $\phi$.
\end{classification}

\begin{proof}
Let $A = C(X) \rtimes_{\alpha} \mathbb{Z}$ and $ B  = C(Y) \rtimes_{\beta} \mathbb{Z}$ be in $\mathcal{A}$ and suppose $\phi :  K_*(A \otimes \mathcal{Z}) \to K_*(B \otimes \mathcal{Z})$ is a graded order isomorphism. Let $q$ be any prime number, let $x \in X$ and $y \in Y$; define the $C^*$-subalgebras  $A_{\{x\}}:= C^{*}(C(X), u C_{0}(X \setminus \{x\}))$ and $B_{\{y\}}:= C^{*}(C(Y), v C_{0}(Y \setminus \{y\}))$, where $u$ and $v$ are the unitaries implementing $\alpha$ and $\beta$ in $C(X) \rtimes_{\alpha} \mathbb{Z}$ and $C(Y) \rtimes_{\beta} \mathbb{Z}$, respectively. It follows from Section 3 of \cite{LinQPhil:KthoeryMinHoms} and Corollary 2.2 of \cite{NgWinter:subhom} that both $A_{\{x\}} \otimes M_{q^{\infty}}$ and $B_{\{y\}} \otimes M_{q^{\infty}}$ have locally finite decomposition rank. By Lemma~\ref{K(A_y x Mq) = K(A x Mq)} above and Theorem 1.2 (4) of \cite{LinQPhil:KthoeryMinHoms}, the $K_0$-groups and tracial state spaces of $A_{\{x\}} \otimes M_{q^{\infty}}$ and $A \otimes M_{q^{\infty}}$ (respectively $B_{\{y\}} \otimes M_{q^{\infty}}$ and $B \otimes M_{q^{\infty}}$) are identical. But then our assumptions on the class $\mathcal{A}$ imply that $A_{\{x\}} \otimes M_{q^{\infty}}$ and $B_{\{y\}} \otimes M_{q^{\infty}}$ have projections separating tracial states. Since $A_{\{x\}} \otimes M_{q^{\infty}}$ and $B_{\{y\}} \otimes M_{q^{\infty}}$ are approximately divisible (cf.\ \cite{Ror:encyc}), we deduce they have real rank zero by the results of \cite{BlaKumRor:apprdiv}, whence tracial rank zero by Theorem 2.1 of \cite{Winter:lfdr}.  Applying Theorem \ref{main theorem} (with $\mathcal{S}$ being the class of finite dimensional $C^{*}$-algebras), $A \otimes M_{q^{\infty}}$ and $B \otimes M_{q^{\infty}}$ have tracial rank zero. Since this is true for any $q$, we may employ Theorem 5.4 of \cite{LinNiu:KKlifting} to conclude that there is a $*$-isomorphism $\Phi : A \otimes \mathcal{Z} \to B \otimes \mathcal{Z}$ inducing $\phi$.  Note that \cite{LinNiu:KKlifting} also requires $A$ and $B$ to satisfy the UCT, which  automatically holds in our situation.
\end{proof}

Most notably, this is the missing link between existing classification results and the work of Toms and the second named author in \cite{TomsWinter:minhom}. There it is shown that if $X$ has finite covering dimension, then the resulting crossed product is $\mathcal{Z}$-stable (\cite{TomsWinter:minhom}, Theorem~4.4); in this case Corollary \ref{classification} becomes Theorem~0.1 of \cite{TomsWinter:minhom}; see also Theorem~A of \cite{TomsWinter:PNAS}. In particular this solves the classification problem for $C^*$-algebras associated to uniquely ergodic minimal finite dimensional dynamical systems.

A compelling set of examples are the $C^*$-algebras arising from minimal homeomorphisms of odd spheres $S^n$ for $n \geq 3$ odd. These were first considered in Section 5 of \cite{Con:Thom} for $n = 3$  and $\alpha$ a minimal diffeomorphism. It follows from Corollary 3 of Section 5 of \cite{Con:Thom} that the crossed product of $C(S^n)$ by $\mathbb{Z}$ induced by a minimal diffeomorphism has no nontrivial projections. The $K_0$-group of such a $C^*$-algebra $A$ is shown in Example 4.6 of \cite{Phi:CancelSRDirLims} to be $K_0(A) = \mathbb{Z}^2$ with order given by $(m, n) \geq 0$ if and only if $n > 0$ or $(m,n) = (0,0)$. These examples are covered by Corollary~\ref{classification} precisely in the uniquely ergodic case (since otherwise \ref{classification} (ii) is not satisfied), so they are classified up to $\mathcal{Z}$-stability. The main result of \cite{TomsWinter:minhom} (which in turn makes heavy use of \cite{Winter:dr-Z-stable}) then shows that $\mathcal{Z}$-stability is automatic, so in the uniquely ergodic case Connes' odd spheres are all isomorphic. The respective statement in the smooth case was already derived in \cite{Winter:dr-Z-stable}, using the inductive limit decomposition of \cite{LinPhi:mindifflimits}. 

To use Corollary~\ref{classification} to classify crossed products with more general tracial state spaces, at the current stage only Lin's classification of $C^{*}$-algebras which are TAI after tensoring with UHF algebras is available \cite{Lin:asu-class}. The algebras covered by this are all rationally Riesz, i.e., their ordered $K_{0}$-groups become Riesz groups after tensoring with the $K$-theory of a UHF algebra. However, for general transformation group $C^{*}$-algebras it is not clear when they are rationally Riesz, and one cannot hope for TAI classification to be a sufficient tool in this case. The strategy would then be to  identify  a suitable class $\mathcal{S}$ of unital $C^{*}$-algebras such that $A_{\{y\}} \otimes M_{q^{\infty}}$ can be verified to be TA$\mathcal{S}$ and such that  TA$\mathcal{S}$ algebras can be classified in a similar manner as TAF or TAI algebras. At least for odd spheres it is not hard to see that they are rationally Riesz, so it only remains to verify that they are indeed TAI up to tensoring with UHF algebras; the result would then be that they are entirely determined by their tracial state spaces,  by virtue of \cite{LinNiu:KKlifting} and our Corollary~\ref{classification}. At least in the case of finitely many ergodic measures we are confident that our results will lead to a complete solution; this will be pursued in a subsequent paper.

We wish to point out that Corollary~\ref{classification} does not require any condition on the dimension of the underlying space and that, without such a condition, classification up to $\mathcal{Z}$-stability is probably the best for which one can hope. The examples constructed by Giol and Kerr in \cite{GioKerr:Subshifts} suggest that counterexamples to the general case of the Elliott conjecture as exhibited by Toms in \cite{Toms:example} can also occur as transformation group $C^{*}$-algebras. One would still expect classification up to $\mathcal{Z}$-stability in this setting, which would then also imply that at least  the crossed products stabilized by  $\mathcal{Z}$ have finite topological dimension.  Conversely, if the underlying space is finite dimensional, then $\mathcal{Z}$-stability is automatic by \cite{TomsWinter:minhom}. In this sense, our result is analogous to \cite{Winter:lfdr}, which in the real rank zero case also provided classification up to $\mathcal{Z}$-stability under otherwise mild structural conditions. 

\section*{Acknowledgments}
We would like to thank the referee for a number of helpful comments. \\
\thanks{{\it Supported by:} EPSRC First Grant EP/G014019/1}.

%
\providecommand{\bysame}{\leavevmode\hbox to3em{\hrulefill}\thinspace}
\providecommand{\MR}{\relax\ifhmode\unskip\space\fi MR }
\providecommand{\MRhref}[2]{%
  \href{http://www.ams.org/mathscinet-getitem?mr=#1}{#2}
}
\providecommand{\href}[2]{#2}
 
\end{document}